\documentclass[11pt,a4paper,twoside]{article}
\usepackage[hmarginratio=1:1,bmargin=1.3in]{geometry}
\usepackage[OT2,T1]{fontenc}
\usepackage{amsmath, amsthm}
\usepackage{amsfonts}
\usepackage{amssymb}
\usepackage[all]{xy}
\usepackage{graphicx}
\usepackage{color}
\usepackage[nottoc, notlof, notlot]{tocbibind}
\usepackage{array}
\usepackage{url}
\usepackage{rotating}
\usepackage{fancyhdr}
\usepackage{fancyhdr}
\usepackage{titlesec}
\usepackage{xfrac} 
\usepackage{multicol}
\title{Homogeneous spaces, algebraic $K$-theory and cohomological dimension of fields}
\author{Diego Izquierdo\\
\small
Max-Planck-Institut f\"{u}r Mathematik \\
\small
\texttt{izquierd@mpim-bonn.mpg.de}
\and
 Giancarlo Lucchini Arteche\\
 \small
 Universidad de Chile\\
  \small
\texttt{luco@uchile.cl} 
 }
 
\date{}

\titleformat{\section}[hang]{\center\Large\bf}{\thesection.}{0.5cm}{}

\theoremstyle{plain}
\newtheorem{theorem}{Theorem}[section]
\newtheorem*{MT}{Main Theorem}
\newtheorem{lemma}[theorem]{Lemma}
\newtheorem{proposition}[theorem]{Proposition}
\newtheorem{corollary}[theorem]{Corollary}
\newtheorem{definition}[theorem]{Definition}

\theoremstyle{definition}
\newtheorem{remarque}[theorem]{Remark}

\newcommand \br {{\rm{Br\,}}}

\newcommand{\M}{\mathcal{M}}
\newcommand{\tors}{\mathrm{TORS}}
\newcommand{\HS}{\mathrm{HS}}
\newcommand{\PHS}{\mathrm{PHS}}
\newcommand{\Red}{\mathrm{Red}}
\newcommand{\ab}{\mathrm{ab}}
\newcommand{\bb}[1]{\mathbb{#1}}

\newcommand \ie {\textit{i.e.~}}

\renewcommand{\ker}{\mathrm{Ker}}

\usepackage{marvosym}

\begin{document}

\maketitle

\begin{abstract}
Let $q$ be a non-negative integer. We prove that a perfect field $K$ has cohomological dimension at most $q+1$ if, and only if, for any finite extension $L$ of $K$ and for any homogeneous space $Z$ under a smooth linear connected algebraic group over $L$, the $q$-th Milnor $K$-theory group of $L$ is spanned by the images of the norms coming from finite extensions of $L$ over which $Z$ has a rational point. We also prove a variant of this result for imperfect fields.\\

\textbf{MSC Classes:} primary 12G10, 19D45; secondary 11E72, 14M17.

\end{abstract}

\section{Introduction}

In 1986, in the article \cite{KK}, Kato and Kuzumaki stated a set of conjectures which aimed at giving a diophantine characterization of cohomological dimension of fields. For this purpose, they introduced some properties of fields which are variants of the classical $C_i$-property and which involve Milnor $K$-theory and projective hypersurfaces of small degree. They hoped that those properties would characterize fields of small cohomological dimension.\\

More precisely, fix a field $L$ and two non-negative integers $q$ and $i$. Let $K_q^M(L)$ be the $q$-th Milnor $K$-group of $L$. For each finite extension $L'$ of $L$, one can define a norm morphism $N_{L'/L}: K_q^M(L')\rightarrow K_q^M(L)$ (see Section 1.7 of \cite{Kat}). Thus, if $Z$ is a scheme of finite type over $L$, one can introduce the subgroup $N_q(Z/L)$ of $K_q^M(L)$ generated by the images of the norm morphisms $N_{L'/L}$ when $L'$ runs through the finite extensions of $L$ such that $Z(L') \neq \emptyset$. One then says that the field $L$ is $C_i^q$ if, for each $n \geq 1$, for each finite extension $L'$ of $L$ and for each hypersurface $Z$ in $\mathbb{P}^n_{L'}$ of degree $d$ with $d^i \leq n$, one has $N_q(Z/L') = K_q^M(L')$. For example, the field $L$ is $C_i^0$ if, for each finite extension $L'$ of $L$, every hypersurface $Z$ in $\mathbb{P}^n_{L'}$ of degree $d$ with $d^i \leq n$ has a 0-cycle of degree 1. The field $L$ is $C_0^q$ if, for each tower of finite extensions $L''/L'/L$, the norm morphism $N_{L''/L'}: K_q^M(L'')\rightarrow K_q^M(L')$ is surjective.\\

Kato and Kuzumaki conjectured that, for $i \geq 0$ and $q\geq 0$, a perfect field is $C_i^q$ if, and only if, it is of cohomological dimension at most $i+q$. This conjecture generalizes a question raised by Serre in \cite{GC} asking whether the cohomological dimension of a $C_i$-field is at most $i$. As it was already pointed out at the end of Kato and Kuzumaki's original paper \cite{KK}, Kato and Kuzumaki's conjecture for $i=0$ follows from the Bloch-Kato conjecture (which has been established by Rost and Voevodsky, cf.~\cite{Riou}): in other words, a perfect field is $C_0^q$ if, and only if, it is of cohomological dimension at most $q$. However, it turns out that the conjectures of Kato and Kuzumaki are wrong in general. For example, Merkurjev constructed in \cite{Mer} a field of characteristic 0 and of cohomological dimension 2 which did not satisfy property $C^0_2$. Similarly, Colliot-Th\'el\`{e}ne and Madore produced in \cite{CTM} a field of characteristic 0 and of cohomological dimension 1 which did not satisfy property $C^0_1$. These counter-examples were all constructed by a method using transfinite induction due to Merkurjev and Suslin. The conjecture of Kato and Kuzumaki is therefore still completely open for fields that usually appear in number theory or in algebraic geometry.\\

In 2015, in \cite{Wit}, Wittenberg proved that totally imaginary number fields and $p$-adic fields have the $C_1^1$ property. In 2018, in \cite{diego}, the first author also proved that, given a positive integer $n$, finite extensions of $\mathbb{C}(x_1,...,x_n)$ and of $\mathbb{C}(x_1,...,x_{n-1})((t))$ are $C_i^q$ for any $i,q \geq 0$ such that $i+q =n$. These are essentially the only known cases of Kato and Kuzumaki's conjectures. In particular, the $C_1^1$ property is still unknown for several usual fields with cohomological dimension 2, such as the field of rational functions $\mathbb{C}((t))(x)$ or the field of Laurent series $\mathbb{C}((x,y))$.\\

In the present article, for each non negative integer $q$, we introduce variants of the $C_1^q$ property and we prove that, contrary to the $C_1^q$ property, they characterize the cohomological dimension of fields. More precisely, we say that a field $L$ is $C_{\HS}^q$ if, for each finite extension $L'$ of $L$ and for each homogeneous space $Z$ under a smooth linear connected algebraic group over $L'$, one has $N_q(Z/L') = K_q^M(L')$. Similarly, we say that a field $L$ is $C_{\PHS}^q$ (resp.~$C_{\Red}^q$) if, for each finite extension $L'$ of $L$ and for each principal homogeneous space $Z$ under a smooth linear connected (resp.~reductive) algebraic group over $L'$, one has $N_q(Z/L') = K_q^M(L')$.  Our main theorem is the following (please refer to section \ref{gener} for the definitions of the cohomological dimension and the separable cohomological dimension):

\begin{MT}
Let $q$ be a non-negative integer. 
\begin{itemize}
\item[(i)]  A perfect field has the $C_{\HS}^q$-property if, and only if, it has cohomological dimension at most $q+1$.
\item[(ii)]  An imperfect field has the $C_{\Red}^q$-property if, and only if, all its finite extensions have separable cohomological dimension at most $q+1$. 
\end{itemize}
\end{MT}

\begin{remarque}
In fact, we will see in section \ref{chs} that the properties $C_\HS^q$ and $C_\Red^q$ are much stronger than what is actually needed to prove that a field has cohomological dimension at most $q+1$. More precisely, we will prove that (cf.~Remark \ref{rem SB}):
\begin{enumerate}
\item A field $K$ has cohomological dimension at most $q+1$ if, for any tower of finite field extensions $M/L/K$ and any element $a\in L^{\times}$, we have $N_q(Z/L)=K_q^M(L)$ for the $L$-variety $Z$ defined by the normic equation $N_{M/L}(\mathbf{x})=a$.
\item A field $K$ has separable cohomological dimension at most $q+1$ if $N_q(Z/L)=K_q^M(L)$ for any positive integer $n$, any finite separable extension $L/K$ and any $\mathrm{PGL}_{n,L}$-torsor $Z$.
\end{enumerate}
\end{remarque}

Our Main Theorem, together with the previous remark, unifies and significantly generalizes several results in the literature:
\begin{itemize}
\item[$\bullet$] The theorems of Steinberg and Springer (see Section III.2.4 of \cite{GC}), which state that, if $K$ is a perfect field with cohomological dimension at most one, then every homogeneous space under a linear connected $K$-group has a zero-cycle of degree 1 (and even a rational point).
\item[$\bullet$] A theorem of Suslin, which states that a field $K$ of characteristic 0 has cohomological dimension at most 2 if, and only if, for any finite extension $L$ of $K$ and any central simple algebra $A$ over $L$, the reduced norm $\mathrm{Nrd}: A^{\times} \rightarrow L^{\times}$ is surjective (see Corollary 24.9 of \cite{suslin} or Theorem 8.9.3 of \cite{GS}).
\item[$\bullet$] A result of Gille which generalizes Suslin's theorem to positive characteristic fields (Theorem 7 of \cite{gille}).
\item[$\bullet$] Two theorems of Wittenberg, which state that $p$-adic fields and totally imaginary number fields have the property $C_{\HS}^1$ (see Corollaries 5.6 and 5.8 of \cite{Wit}).
\end{itemize}

Since our result applies to all fields, our proof has to be purely algebraic and/or geometric: contrary to the cases of $p$-adic fields and number fields dealt by Wittenberg, we have to systematically avoid arithmetical arguments. In section \ref{gener}, we gather some generalities about the cohomological dimension of fields. In particular, we recall the ``good'' definition of the cohomological dimension for positive characteristic fields as well as a characterization of this invariant in terms of Milnor $K$-theory. In section \ref{chs}, we settle the easy direction of our main theorem: in other words, we prove that a field $L$ having the $C_{\Red}^q$-property has separable cohomological dimension at most $q+1$. The core of the article is section \ref{main}, where we prove our main theorem in three steps:

\begin{itemize}
\item[(1)] We first prove that, if $P$ is a finite Galois module over a characteristic 0 field $L$ of cohomological dimension $\leq q+1$ and $\alpha$ is an element in $H^2(L,P)$, then $K_q^M(L)$ is spanned by the images of the norms coming from finite extensions of $L$ trivializing $\alpha$. This requires to use the Bloch-Kato conjecture but also some properties of norm varieties that have been established by Rost and Suslin.
\item[(2)] We then use a theorem of Steinberg and step (1) to deal with principal homogeneous spaces over characteristic zero fields;  the case of principal homogeneous spaces over positive characteristic fields is then solved by reducing to the characteristic 0 case.
\item[(3)] We finally deal with the case of homogeneous spaces over perfect fields by using a theorem of Springer in non-abelian cohomology which reduces us to the case with finite solvable stabilizers. We deal with this last case by using ``d\'evissage'' techniques in non-abelian cohomology and step (2).
\end{itemize}

Given a field $L$, if we denote by $\mathrm{cd}(L)$ and by $\mathrm{sd}(L)$ the cohomological dimension and the separable cohomological dimension of $L$ respectively, the following diagram summarizes the implications we settle along the proof:

 \begin{multicols}{2}
 \centering
Case when $L$ is perfect:
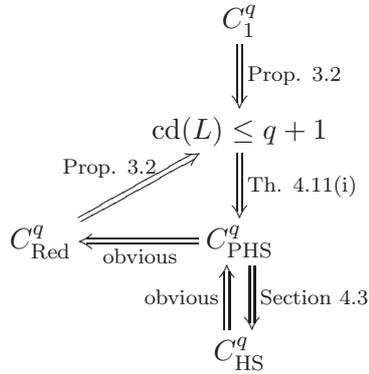
\[\xymatrix{
&C_1^q \ar@{=>}[d]^{\text{Prop. \ref{onedir}}} \\ 
&\mathrm{cd}(L) \leq q+1 \ar@{=>}[d]^{\text{Th. \ref{poscar}(i)}} \\
C^q_{\Red} \ar@{=>}[ur]^{\text{Prop. \ref{onedir}}} &C^q_{\PHS} \ar@{=>}@<1ex>[d]^{\text{Section \ref{last}}} \ar@{=>}[l]^{\text{obvious}} \\
&C^q_{\HS} \ar@{=>}@<1ex>[u]^{\text{obvious}}
}\]

Case when $L$ is imperfect:
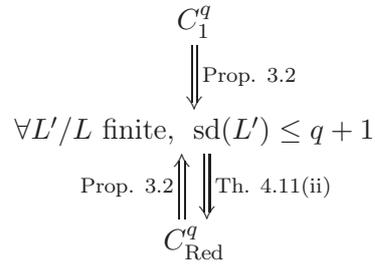
\[\xymatrix{
C_1^q \ar@{=>}[d]^{\text{Prop. \ref{onedir}}}\\
\forall L'/L\text{ finite, }\;\mathrm{sd}(L') \leq q+1 \ar@{=>}@<1ex>[d]^{\text{Th. \ref{poscar}(ii)}}\\
C^q_{\Red} \ar@{=>}@<1ex>[u]^{\text{Prop. \ref{onedir}}}
}\]
\end{multicols}

\subsection*{Preliminaries on Milnor $K$-theory}

Let $L$ be any field and let $q$ be a non-negative integer. The $q$-th Milnor $K$-group of $L$ is by definition the group $K_0^M(L)=\mathbb{Z}$ if $q =0$ and:
$$K_q^M(L):= \underbrace{L^{\times} \otimes_{\mathbb{Z}} ... \otimes_{\mathbb{Z}} L^{\times}}_{q \text{ times}} / \left\langle x_1 \otimes ... \otimes x_q | \exists i,j, i\neq j, x_i+x_j=1 \right\rangle$$
if $q>0$. For $x_1,...,x_q \in L^{\times}$, the symbol $\{x_1,...,x_q\}$ denotes the class of $x_1 \otimes ... \otimes x_q$ in $K_q^M(L)$. More generally, for $r$ and $s$ non-negative integers such that $r+s=q$, there is a natural pairing:
$$K_r^M(L) \times K_s^M(L) \rightarrow K_q^M(L)$$
which we will denote $\{\cdot, \cdot\}$.\\

When $L'$ is a finite extension of $L$, one can construct a norm homomorphism
\[N_{L'/L}: K_q^M(L') \rightarrow K_q^M(L),\]
satisfying the following properties (see Section 1.7 of \cite{Kat} or section 7.3 of \cite{GS}):
\begin{itemize}
\item[$\bullet$] For $q=0$, the map $N_{L'/L}: K_0^M(L') \rightarrow K_0^M(L)$ is given by multiplication by $[L':L]$.
\item[$\bullet$] For $q=1$, the map $N_{L'/L}: K_1^M(L') \rightarrow K_1^M(L)$ coincides with the usual norm $L'^{\times} \rightarrow L^{\times}$.
\item[$\bullet$] If $r$ and $s$ are non-negative integers such that $r+s=q$, we have $N_{L'/L}(\{x,y\})=\{x,N_{L'/L}(y)\}$ for $x \in K_r^M(L)$ and $y\in K_s^M(L')$.
\item[$\bullet$] If $L''$ is a finite extension of $L'$, we have $N_{L''/L} = N_{L'/L} \circ N_{L''/L'}$.
\end{itemize}
\vspace{13pt}
For each $L$-scheme of finite type, we denote by $N_q(Z/L)$ the subgroup of $K_q^M(L)$ generated by the images of the maps $N_{L'/L}: K_q^M(L') \rightarrow K_q^M(L)$ when $L'$ runs through the finite extensions of $L$ such that $Z(L')\neq \emptyset$. In particular, $N_0(Z/L)$ is the subgroup of $\mathbb{Z}$ generated by the index of $Z$ (\ie the gcd of the degrees $[L':L]$ when $L'$ runs through the finite extensions of $L$ such that $Z(L') \neq \emptyset$). 

\section{Generalities on the cohomological dimension}\label{gener}

We start the article by recalling the ``good'' definition of the cohomological dimension for fields of any characteristic:

\begin{definition}
Let $K$ be any field. 
\begin{itemize}
\item[(i)] Let $\ell$ be a prime number different from the characteristic of $K$. The $\ell$-cohomological dimension $\mathrm{cd}_{\ell}(K)$ and the separable $\ell$-cohomological dimension $\mathrm{sd}_{\ell}(K)$ of $K$ are both the $\ell$-cohomological dimension of the absolute Galois group of $K$.
\item[(ii)] (Kato, \cite{Kato}; Gille, \cite{gille}). Assume that $K$ has characteristic $p>0$. Let $\Omega^i_K$ be the $i$-th exterior product over $K$ of the absolute differential module $\Omega^1_{K/\mathbb{Z}}$ and consider the morphism $\mathfrak{p}^i_K\,:\, \Omega^i_K \to \Omega^i_K/d(\Omega^{i-1}_K)$ defined by
\[x \frac{dy_1}{y_1} \wedge ... \wedge \frac{dy_i}{y_i} \mapsto (x^p-x) \frac{dy_1}{y_1} \wedge ... \wedge \frac{dy_i}{y_i} \mod d(\Omega^{i-1}_K),\]
for $x\in K$ and $y_1,...,y_i \in K^{\times}$: this morphism is well-defined by sections 9.2 and 9.4 of \cite{GS}. Let $H^{i+1}_p(K)$ be the cokernel of $\mathfrak{p}^i_K$. The $p$-cohomological dimension $\mathrm{cd}_{p}(K)$ of $K$ is the smallest integer $i$ (or $\infty$ if such an integer does not exist) such that $[K:K^p]\leq p^i$ and $H^{i+1}_p(L)=0$ for all finite extensions $L$ of $K$. The separable $p$-cohomological dimension $\mathrm{sd}_{p}(K)$ of $K$ is the smallest integer $i$ (or $\infty$ if such an integer does not exist) such that $H^{i+1}_p(L)=0$ for all finite separable extensions $L$ of $K$.
\item[(iii)] The cohomological dimension $\mathrm{cd}(K)$ of $K$ is the supremum of all the $\mathrm{cd}_{\ell}(K)$'s when $\ell$ runs through all prime numbers. The separable cohomological dimension $\mathrm{sd}(K)$ of $K$ is the supremum of all the $\mathrm{sd}_{\ell}(K)$'s when $\ell$ runs through all prime numbers. 
\end{itemize}
\end{definition}

The following proposition is probably well known to experts, but we didn't find an appropriate reference covering the positive characteristic case:

\begin{proposition}\label{cohdim}
Let $q$ be a non-negative integer, let $\ell$ be a prime number and let $K$ be any field. 
\begin{itemize}
\item[(i)] Assume that, for any tower of finite extensions $M/L/K$, the cokernel of the norm $N_{M/L}: K_q^M(M)\rightarrow K_q^M(L)$ has no $\ell$-torsion. Then $K$ has $\ell$-cohomological dimension at most $q$.
\item[(ii)] Assume that, for any tower of finite separable extensions $M/L/K$, the cokernel of the norm $N_{M/L}: K_q^M(M)\rightarrow K_q^M(L)$ has no $\ell$-torsion. Then $K$ has separable $\ell$-cohomological dimension at most $q$.
\end{itemize}
\end{proposition}

\begin{proof}
The proofs of (i) and (ii) being very similar, we only prove (i). \\

First assume that $\ell$ is different from the characteristic of $K$. In that case, the proposition is essentially a consequence of Lemma 7 of \cite{KK}. But we are going to give a full proof since that lemma is only stated without proof. To do so, consider a finite extension $L$ of $K$ and a symbol $\{a_1, ..., a_{q+1}\} \in K_{q+1}^M(L)$. Let $M$ be the splitting field of the polynomial $T^{\ell}-a_{q+1}$. Since the cokernel of $N_{M/L}: K_q^M(M)\rightarrow K_q^M(L)$ has no $\ell$-torsion, we can find $b \in K_q^M(M)$ such that $\{a_1, ..., a_{q}\} \equiv N_{M/L}(b) \mod \ell K_q^M(L)$. Hence $\{a_1, ..., a_{q+1}\} \equiv N_{M/L}(\{b,a_{q+1}\})\mod \ell K_{q+1}^M(L) $ and $\{a_1, ..., a_{q+1}\} \in \ell K_{q+1}^M(L) $. By the Bloch-Kato conjecture, we deduce that the group $H^{q+1}(L,\mu_{\ell}^{\otimes (q+1)})$ is trivial. This being true for any finite extension $L$ of $K$, Corollary I.3.3.1 and Proposition I.4.1.21 of \cite{GC} then imply that $K$ has $\ell$-cohomological dimension at most $q$.\\

Now assume that $\ell$ is the characteristic of $K$. We first prove that $H^{q+1}_{\ell}(L)=0$ for all finite extensions $L$ of $K$. We therefore consider such a finite extension $L$ of $K$ and an element $x \frac{dy_1}{y_1} \wedge ... \wedge \frac{dy_{q}}{y_{q}}$ of $\Omega^{q}_L$ with $x\in L$ and $y_1,...,y_{q} \in L^{\times}$. Let $M$ be the splitting field of the polynomial $T^\ell-T-x$ over $L$. By assumption, we can find an element $z\in K_q^M(M)$ such that $\{y_1,...,y_q\} \equiv N_{M/L}(z) \mod \ell K_q^M(L)$. Now recall the Bloch-Gabber-Kato theorem, which states that, if $\nu(q)_L$ denotes the kernel of $\mathfrak{p}^q_L$, the differential symbol:
\begin{align*}
\psi_L^q: K_q^M(L)/\ell K_q^M(L) & \to \nu(q)_L\\
\{a_1,...,a_q\} & \mapsto \frac{da_1}{a_1} \wedge ... \wedge \frac{da_q}{a_q},
\end{align*}
is an isomorphism (Theorem 9.5.2 of \cite{GS}). Through this isomorphism, the norm on Milnor $K$-theory corresponds to the trace on modules of differential $q$-forms (Lemma 9.5.7 of \cite{GS}). Hence we can find an element $z'\in \nu(q)_M$ such that $\frac{dy_1}{y_1} \wedge ... \wedge \frac{dy_q}{y_q} = \mathrm{tr}_{M/L}(z')$. By our choice of the field $M$, we know that the class of $xz'$ modulo $d(\Omega^{q-1}_M)$ is in the image of $\mathfrak{p}_M^q$. We can therefore conclude that the class of $x \frac{dy_1}{y_1} \wedge ... \wedge \frac{dy_{q}}{y_{q}}$ modulo $d(\Omega^{q-1}_L)$ is in the image of $\mathfrak{p}_L^q$ by observing that the following diagram is commutative:
\[\xymatrix{
\Omega^q_M \ar[r]^-{\mathfrak{p}^q_M} \ar[d]_{\mathrm{tr}_{M/L}} & \Omega^q_M/d(\Omega^{q-1}_M) \ar[d]^{\mathrm{tr}_{M/L}} \\
\Omega^q_L \ar[r]^-{\mathfrak{p}^q_L}  & \Omega^q_L/d(\Omega^{q-1}_L).
}\]

We now prove that $[K:K^\ell]\leq \ell^q$. Assume the contrary, so that we can find $x_1,...,x_{q+1}$ a family of $q+1$ elements of $K$ which are $\ell$-independent (in the sense of section 26 of \cite{Mat}). By assumption, we can find $w\in K_q^M(K(\sqrt[\ell]{x_{q+1}}))$ such that $$N_{K(\sqrt[\ell]{x_{q+1}})/K}(w)=\{x_1,...,x_q\}.$$ We then also have $N_{K^s(\sqrt[\ell]{x_{q+1}})/K^s}(w)=\{x_1,...,x_q\}$, where $K^s$ denotes a separable closure of $K$. Since all finite extensions of $K^s$ have degree a power of $\ell$, Corollary 7.2.10 of \cite{GS} implies that $K_q^M(K^s(\sqrt[\ell]{x_{q+1}}))$ is spanned, as an abelian group, by elements of the form $\{ b_1,...,b_q \}$ with $b_2,...,b_{q}\in K^s$ and $b_1 \in K^s(\sqrt[\ell]{x_{q+1}})$. The image of such a symbol $\{ b_1,...,b_q \}$ by $N_{K^s(\sqrt[\ell]{x_{q+1}})/K^s}$ is the symbol $\{ N_{K^s(\sqrt[\ell]{x_{q+1}})/K^s}(b_1),b_2,...,b_q \}$, where $N_{K^s(\sqrt[\ell]{x_{q+1}})/K^s}(b_1) \in (K^s)^\ell(x_{q+1})$. Hence $\{x_1,...,x_q\}$ lies in the subgroup of $K_q^M(K^s)$ spanned by elements of the form $\{c_1,...,c_q\}$ with $c_1 \in (K^s)^\ell(x_{q+1})$. In other words, by the Bloch-Gabber-Kato theorem, $\frac{dx_1}{x_1} \wedge ... \wedge \frac{dx_q}{x_q}$ lies in the sub-$K^s$-vector space of $\Omega^q_{K^s}$ spanned by the $dx_{q+1} \wedge \omega$ with $\omega\in \Omega^{q-1}_{K^s}$. By Theorem 26.5 of \cite{Mat}, this contradicts the $\ell$-independence of the family $x_1,...,x_{q+1}$.
\end{proof}

\section{The $C_{\HS}^q$, $C_{\PHS}^q$ and $C_{\Red}^q$ properties} \label{chs}

Let $q$ be a non-negative integer. In the article \cite{KK}, Kato and Kuzumaki introduce the notion of $C_1^q$ fields: they say that a field $K$ satisfies the $C_1^q$ property if, for every finite extension $L$ of $K$ and for every hypersurface $Z$ in $\mathbb{P}^n_{L}$ of degree $d$ with $d \leq n$, we have $N_q(Z/L)=K_q^M(L)$. In this article, we are interested in the following variants of this property:

\begin{definition}
Let $q$ be a non-negative integer. We say that a field $K$ has the $C_{\HS}^q$ property if, for each finite extension $L$ of $K$ and for each homogeneous space $Z$ under a smooth linear connected algebraic group over $L$, one has $N_q(Z/L) = K_q^M(L)$. Similarly, we say that a field $K$ has the $C_{\PHS}^q$ property (resp. the $C_{\Red}^q$ property) if, for each finite extension $L$ of $K$ and for each principal homogeneous space $Z$ under a smooth linear connected (resp.~reductive) algebraic group over $L$, one has $N_q(Z/L) = K_q^M(L)$.
\end{definition}

Of course, the $C_{\HS}^q$ property implies the $C_{\PHS}^q$ property, which itself implies the $C_{\Red}^q$ property. The following proposition shows that a field satisfying those properties has small cohomological dimension:

\begin{proposition}\label{onedir}
Let $K$ be a field.
\begin{itemize}
\item[(i)] Assume that, for any tower of finite (resp. finite separable) extensions $M/L/K$ and any element $a \in L^{\times}$, we have $N_q(Z/L)=K_q^M(L)$ for the $L$-variety:
\[Z: N_{M/L}(\mathbf{x}) = a.\]
Then $K$ has cohomological dimension (resp. separable cohomological dimension) at most $q+1$.
\item[(ii)] Assume that, for any finite separable extension $L/K$ and any Severi-Brauer $L$-variety $Z$, we have $N_q(Z/L)=K_q^M(L)$. Then $K$ has separable cohomological dimension at most $q+1$.
\end{itemize}

\end{proposition}

\begin{proof}[Proof]
Let's prove (i) first. The statements about the cohomological dimension and the separable cohomological dimension can be proved in the same way, so we will only deal with the cohomological dimension. By Proposition \ref{cohdim}, we only need to prove that for any tower of finite field extensions $M/L/K$, the norm $N_{M/L}: K_{q+1}^M(M) \rightarrow K_{q+1}^M(L)$ is surjective. To do so, consider a symbol $\{a_1,...,a_{q+1}\} \in K_{q+1}^M(L)$. Also consider the following variety:
\[Z: N_{M/L}(\mathbf{x}) = a_1.\]
By assumption, one can find finite extensions $L_1,...,L_r$ of $L$ such that:
$$\begin{cases} 
\forall i \in \{1,...,r\}, Z(L_i)\neq \emptyset \\
\{a_2,...,a_{q+1}\} \in \left\langle N_{L_i/L}(K_q^M(L_i)) \mid 1 \leq i \leq r \right\rangle .
 \end{cases} $$
 In other words:
 $$\begin{cases} 
\forall i \in \{1,...,r\}, \; \exists x_i \in (M \otimes_L L_i)^{\times}, \; N_{M \otimes_L L_i/L_i}(x_i) = a_1 \\
\exists (y_1,...,y_r) \in \prod_{i=1}^r K_q^M(L_i), \; \{a_2,...,a_{q+1}\} = \prod_{i=1}^r N_{L_i/L}(y_i).
 \end{cases} $$
 Hence:
 \begin{align*}
N_{M/L}\left( \prod_{i=1}^r N_{M \otimes L_i /M} \left( \{x_i,y_i\} \right) \right) & = \prod_{i=1}^r N_{M \otimes L_i /L}\left( \{x_i,y_i\} \right) \\
& = \prod_{i=1}^r N_{L_i/L}\left(N_{M \otimes L_i /L_i} \left( \{x_i,y_i \}  \right) \right)\\
& = \prod_{i=1}^r N_{L_i/L}\left( \{N_{M \otimes L_i /L_i} (x_i),y_i \} \right)  \\
& = \prod_{i=1}^r N_{L_i/L}\left( \{a_1,y_i \} \right)  \\
& =  \{ a_1,\prod_{i=1}^r N_{L_i/L}(y_i) \}    \\
& =  \{ a_1,a_2,...,a_{q+1} \} 
 \end{align*}
 and $\{ a_1,a_2,...,a_{q+1} \}$ is indeed  in the image of $N_{M/L}: K_{q+1}^M(M) \rightarrow K_{q+1}^M(L)$.\\

Let's now prove (ii). To do so, fix a prime number $\ell$ and first assume that $\ell$ is different from the characteristic of $K$. Consider a finite extension $L$ of $K$ containing a primitive $\ell$-th root of unity and a symbol $\{a_1,...,a_{q+2}\} \in H^{q+2}(L,\mathbb{Z}/\ell\mathbb{Z})$. By assumption and by the Bloch-Kato conjecture, one can find finite extensions $L_1,...,L_r$ of $L$ and elements $b_1\in K_{q}(L_1),...,b_r\in K_{q}(L_r)$ such that $\{a_1,a_2\}|_{L_i}=0$ for each $i$ and $\{a_3,...,a_{q+2}\} = \sum_{i=1}^r \mathrm{Cores}_{L_i/L}(b_i)$. Hence $\{a_1,...,a_{q+2}\} = \sum_{i=1}^r \mathrm{Cores}_{L_i/L}(\{a_1,a_2\} \cup b_i)=0$, and the group $H^{q+2}(L,\mathbb{Z}/\ell\mathbb{Z})$ is trivial. This being true for any finite extension $L$ of $K$ containing a primitive $\ell$-th root of unity, Corollary I.3.3.1 and Proposition I.4.1.21 of \cite{GC} then imply that $K$ has $\ell$-cohomological dimension at most $q+1$.

Let's finally assume that $K$ has characteristic $\ell$. Fix a finite separable extension $L$ of $K$ and an element $x \frac{dy_1}{y_1} \wedge ... \wedge \frac{dy_{q+1}}{y_{q+1}}$ of $\Omega^{q+1}_L$ with $x\in L$ and $y_1,...,y_{q+1} \in L^{\times}$. Consider the cyclic central simple algebra $\alpha \in {_{\ell}}\mathrm{Br}\, L$ corresponding to the form $x \frac{dy_1}{y_1}$ through the isomorphism of Theorem 9.2.4 of \cite{GS}. By assumption, we can find finite extensions $L_1,...,L_r$ of $L$  and elements $z_1\in K_q^M(L_1)$, ..., $z_r\in K_q^M(L_r)$ such that $\alpha|_{L_i}=0$ for each $i$ and $\{y_2,...,y_{q+1}\} =\prod_i^r N_{L_i/L}(z_i)$. Hence, by the Bloch-Gabber-Kato theorem (Theorem 9.5.2 of \cite{GS}), we can find elements $z'_1\in \nu(q)_{L_1},...,z'_r\in \nu(q)_{L_r}$ such that $\frac{dy_2}{y_2} \wedge ... \wedge \frac{dy_{q+1}}{y_{q+1}} = \sum_{i=1}^r \mathrm{tr}_{L_i/L}(z'_i)$. We then have $x \frac{dy_1}{y_1} \wedge ... \wedge \frac{dy_{q+1}}{y_{q+1}}= \sum_{i=1}^r \mathrm{tr}_{L_i/L}(x \frac{dy_1}{y_1} \wedge z'_i) \in \mathrm{Im}(\mathfrak{p}^{q+1}_L)$. This proves that $H^{q+2}_{\ell}(L)=0$, as wished.
\end{proof}

\begin{remarque}\label{rem C1}
Since the varieties considered in Proposition \ref{onedir}(i) are, up to ho\-mo\-ge\-ni\-za\-tion, hypersurfaces of degree $d=[M:L]$ in $\bb P^d_L$, this shows that a $C_1^q$ field has cohomological dimension at most $q+1$. This fact seems to have been overlooked in the literature: for instance, in Theorem 4.2 of \cite{Wit}, the assumption concerning the $C_0^q$ property is unnecessary.
\end{remarque}

\begin{remarque}\label{rem SB}
As stated in the introduction, Proposition \ref{onedir} shows that, in order to check that a field has separable cohomological dimension at most $q+1$, it suffices to check one of the following conditions:
\begin{itemize}
\item[(a)] For each finite extension $L$ of $K$ and each torsor $Z$ under a normic torus over $L$, we have $N_q(Z/L)=K_q^M(L)$.
\item[(b)] For each positive integer $n$, each finite extension $L$ of $K$ and each $\mathrm{PGL}_{n,L}$-torsor $Z$, we have $N_q(Z/L)=K_q^M(L)$.
\end{itemize} 
In particular, a field having the $C_{\Red}^q$ property has separable cohomological dimension at most $q+1$.
\end{remarque}

\section{Proof of the Main Theorem}\label{main}

This section is devoted to the proof of the following theorem, which corresponds to the ``difficult'' direction of our Main Theorem:

\begin{theorem}\label{reduction}
Let $q$ be a non-negative integer. Let $K$ be a any field.  
\begin{itemize}
\item[(i)] If $K$ is perfect and has cohomological dimension at most $q+1$, then for any homogeneous space $Z$ under a smooth linear connected group over $K$, we have $N_q(Z/K)=K_q^M(K)$.
\item[(ii)] If $K$ is imperfect and all its finite extensions have separable cohomological dimension at most $q+1$, we have $N_q(Z/K)=K_q^M(K)$ for any principal homogeneous space $Z$ under a (smooth, connected) reductive group over $K$.
\end{itemize}
\end{theorem}

Note that our Main Theorem immediately follows from Proposition \ref{onedir}, Theorem \ref{reduction} and the fact that cohomological dimension, when it is finite, is preserved under finite separable extensions.

\subsection{First step: trivializing Galois cohomology classes}

In this first step of the proof of theorem \ref{reduction}, we are interested in the following variant of the groups $N_q(Z/K)$:

\begin{definition}
Let $K$ be a characteristic zero field and let $P$ be a finite Galois module over $K$. Fix a cohomology class $\alpha \in H^2(K,P)$. We define $N_q(\alpha/K)$ as the subgroup of $K_q^M(K)$ spanned by the images of the norms coming from finite extensions $L$ of $K$ such that $\alpha|_L=0 \in H^2(L,P)$.
\end{definition}

The main result we are going to prove in this context is the following:

\begin{theorem}\label{galcoh}
Let $q$ be a non-negative integer and let $K$ be a field of characteristic 0. Fix an algebraic closure $\overline{K}$. Assume that we are given a Galois extension $K_{\infty}$ of $K$ in $\overline{K}$ such that:
\begin{itemize}
\item[(A)] for each finite extension $L$ of $K$ and each prime number $\ell$, the morphism $$H^{q+2}(L,\mu_{\ell}^{\otimes (q+1)}) \rightarrow H^{q+2}(LK_{\infty},\mu_{\ell}^{\otimes (q+1)}) $$ is injective;
\item[(B)] for each finite extension $L$ of $K$ and each prime number $\ell$ such that $L$ contains a primitive $\ell$-th root of unity, the group ${_{\ell}}\br(LK_{\infty}/ L)$ is spanned by cyclic central algebras.
\end{itemize}
  Then for any finite Galois module $P$ over $K$ which becomes diagonalisable over $K_{\infty}$ and any cohomology class $$\alpha \in \ker\left( H^2(K,P) \rightarrow H^2(K_{\infty},P)\right),$$ we have $N_q(\alpha/K)=K_q^M(K)$.
\end{theorem}

By taking $K_{\infty} = \overline{K}$ and by applying the Merkurjev-Suslin theorem (\cite{Riou}), we obtain the following corollary:

\begin{corollary}\label{galcoh2}
Let $q$ be a non-negative integer and let $K$ be a field of characteristic 0 and of cohomological dimension at most $q+1$. For any finite Galois module $P$ over $K$ and any cohomology class $\alpha \in  H^2(K,P),$ we have $N_q(\alpha/K)=K_q^M(K)$.
\end{corollary}

\begin{remarque}
In order to prove Theorem \ref{reduction} for perfect fields, we will only need the previous corollary. But to deal with imperfect fields, we will need to consider other choices for $K_{\infty}$ and hence we will have to use the more general Theorem \ref{galcoh}.
\end{remarque}

\subsubsection{The Galois module $\mathbb{Z}/\ell\mathbb{Z}$} \label{firststep1}

\begin{proposition}
Let $q$ be a non-negative integer, let $\ell$ be a prime number and let $K$ be a field of characteristic 0. Consider $q+2$ elements $a_1,...,a_{q+2}$ in $K^{\times}$ such that the symbol $\{a_1,...,a_{q+2}\}$ is trivial in $K_{q+2}^M(K)/\ell$. Set $\alpha := \{a_1,a_2\} \in H^2(K,\mu_{\ell}^{\otimes 2})$. Then $\{a_3,...,a_{q+2}\} \in N_q(\alpha/K)$.
\end{proposition}

\begin{proof}[Proof]
We proceed by induction on $q$, noting that if $q=0$, there is nothing to prove. Assume then that the proposition is true for some integer $q\geq 0$ and consider $q+3$ elements $a_1,...,a_{q+3}$ in $K^{\times}$ such that the symbol $\{a_1,...,a_{q+3}\}$ is trivial in $K_{q+3}^M(K)/\ell$. Observe that, if $\alpha=0$, there is nothing to prove. Hence we may and do assume that $\alpha\neq 0$.

Let $K_{\ell}$ be the field fixed by an $\ell$-Sylow subgroup of $\mathrm{Gal}(\overline{K}/K)$. Since $\alpha\neq 0$, a restriction-corestriction argument shows that $\alpha|_{K_{\ell}}\neq 0$. Let $m\in \{2,...,q+2\}$ be the largest integer such that $\{a_1,...,a_m\}|_{K_{\ell}} \neq 0 \in H^m(K_{\ell},\mu_{\ell}^{\otimes m})$. According to Theorem 1.21 of \cite{sj}, there exists a geometrically irreducible projective $\ell$-generic splitting $\nu_{m-1}$-variety $X$ for $\{a_1,...,a_m\}|_{K_{\ell}}$ (see Definitions 1.10 and 1.20 of \cite{sj}). Moreover, by Theorem A.1 of \cite{sj}, we have an exact sequence:
$$\bigoplus_{x\in X\text{ closed}} K_{\ell}(x)^{\times} \xrightarrow[]{\bigoplus N_{K_{\ell}(x)/K_{\ell}}} K_{\ell}^{\times} \xrightarrow[]{\{a_1,...,a_m\}\cup} K_{m+1}^M(K_{\ell})/\ell.$$
Since $\{a_1,...,a_{m+1}\}|_{K_{\ell}} = 0 \in H^{m+1}(K_{\ell},\mu_{\ell}^{\otimes m+1})$, we deduce that there are $r$ closed points $x_1,...,x_r$ in $X$ and an element $b_i\in K_{\ell}(x_i)^{\times}$ for each $i\in \{1,...,r\}$ such that $a_{m+1} = \prod_{i=1}^r N_{K_{\ell}(x_i)/K_{\ell}}(b_i)$.

Now, for each $i$, the symbol $\{a_1,...,a_{m}\}|_{K_{\ell}(x_i)}$ is trivial in $H^{m}(K_{\ell}(x_i),\mu_{\ell}^{\otimes m})$. By the inductive assumption, we deduce that $\{a_3,...,a_{m}\} \in N_{m-2}(\alpha|_{K_{\ell}(x_i)}/K_{\ell}(x_i))$ for each $i$. In other words, we can find some finite extensions $K_{i,1},...,K_{i,r_i}$ of $K_{\ell}(x_i)$ and an $r_i$-tuple $(c_{i,1},...,c_{i,r_i}) \in \prod_{j=1}^{r_i} K_{m-2}^M(K_{i,j})$ such that:
$$\begin{cases} 
\alpha|_{K_{i,j}}=0\,\text{ for each }\,j\in \{1,...,r_i\}, \\
\{a_3,...,a_{m}\} = \prod_{j=1}^{r_i} N_{K_{i,j}/K_\ell(x_i)}(c_{i,j}).
 \end{cases} $$ 
 We then compute:
 \begin{align*}
 \prod_{i=1}^{r} \prod_{j=1}^{r_i} N_{K_{i,j}/K_{\ell}}(&\{c_{i,j},b_i,a_{m+2},...,a_{q+3}\}) \\ &=  \prod_{i=1}^{r} \prod_{j=1}^{r_i} N_{K_{\ell}(x_i)/K_{\ell}}\left( N_{K_{i,j}/K_\ell(x_i)}(\{c_{i,j},b_i,a_{m+2},...,a_{q+3}\})\right)\\
 &=  \prod_{i=1}^{r} N_{K_{\ell}(x_i)/K_{\ell}}\left( \{a_3,...,a_m,b_i,a_{m+2},...,a_{q+3}\}\right) \\
 &=  \{a_3,...,a_{q+3}\}.
 \end{align*}
 Hence $\{a_3,...,a_{q+3}\}|_{K_{\ell}} \in N_{q+1}(\alpha|_{K_{\ell}}/{K_{\ell}})$. This means that there exists a finite extension $K'$ of $K$ contained in $K_{\ell}$ such that $\{a_3,...,a_{q+3}\}|_{K'} \in N_{q+1}(\alpha|_{K'}/{K'})$. Since $\ell$ does not divide the degree of the extension $K'/K$, we deduce that $\{a_3,...,a_{q+3}\} \in N_{q+1}(\alpha/K)$.
\end{proof}

As a corollary, we obtain the following particular case of Theorem \ref{galcoh}:

\begin{corollary}\label{sevbr}
Let $q$ be a non-negative integer and let $\ell$ be a prime number. Let $K$ be a field of characteristic 0. Fix an algebraic closure $\overline{K}$. Assume that we are given a Galois extension $K_{\infty}$ of $K$ in $\overline{K}$ satisfying assumptions (A) and (B) of Theorem \ref{galcoh}. Then for any $\alpha \in \mathrm{Ker}\left( H^2(K,\mathbb{Z}/\ell\mathbb{Z}) \rightarrow H^2(K_{\infty},\mathbb{Z}/\ell\mathbb{Z})   \right)$, we have $N_q(\alpha/K)=K_q^M(K)$.
\end{corollary}

\begin{proof}[Proof]
If $K$ contains a primitive $\ell$-th root of unity, this follows immediately from the previous proposition, the fact that $\alpha|_{K_{\infty}}=0$ and assumptions (A) and (B) of Theorem \ref{galcoh}. We may hence assume that $K$ does not contain any primitive $\ell$-th root of unity. In that case, let $\zeta_{\ell}$ be a primitive $\ell$-th root of unity in $\overline{K}$. Then we know that $N_q(\alpha|_{K(\zeta_{\ell})}/K(\zeta_{\ell}))=K_q^M(K(\zeta_{\ell}))$, so that $N_q(\alpha/K)$ contains $[K(\zeta_{\ell}):K]K_q^M(K)$. Moreover, for each prime number $p$ different from $\ell$, if $L$ is a finite Galois extension of $K$ such that $\alpha|_L=0$ and if $L_p$ is a field fixed by a $p$-Sylow subgroup of $\mathrm{Gal}(L/K)$, then a restriction-corestriction argument shows that $\alpha|_{L_p}=0$, so that $N_q(\alpha/K)$ contains $[L_p:K]K_q^M(K)$. We deduce that $N_q(\alpha/K)=K_q^M(K)$.
\end{proof}

\subsubsection{Behaviour with respect to ``d\'evissages''}

In order to reduce Theorem \ref{galcoh} to the case studied in the previous paragraph, we will need to carry out some ``d\'evissages''. The following easy lemma will be very useful for that purpose:

\begin{lemma}\label{easy}
Let $q$ be a non-negative integer. Let $K$ be a field of characteristic 0, fix an algebraic closure $\overline{K}$ and consider an exact sequence of finite Galois modules:
$$0 \rightarrow P \rightarrow Q \rightarrow R \rightarrow 0.$$
Let $K_{\infty}$ be an algebraic extension of $K$ in $\overline{K}$ such that, for each finite extension $L$ of $K$, the morphism $H^2(LK_{\infty},P) \rightarrow H^2(LK_{\infty},Q)$ is injective. Make the following assumptions:
\begin{itemize}
\item[(i)] for any $\beta \in \mathrm{Ker}\left( H^2(K,R) \rightarrow H^2(K_{\infty},R)\right)$, we have $N_q(\beta/K) = K_q^M(K)$;
\item[(ii)] for any finite extension $L$ of $K$ and for any $\gamma \in \mathrm{Ker}\left( H^2(L,P) \rightarrow H^2(LK_{\infty},P)\right)$, we have $N_q(\gamma/L) = K_q^M(L)$.
\end{itemize}
Then, for any $\alpha\in \mathrm{Ker}\left( H^2(K,Q) \rightarrow H^2(K_{\infty},Q)\right)$, we have $N_q(\alpha/K) = K_q^M(K)$.
\end{lemma}

\begin{proof}
Let $\alpha$ be an element of $ \mathrm{Ker}\left( H^2(K,Q) \rightarrow H^2(K_{\infty},Q)\right)$ and take any $x \in K_q^M(K)$. Let $\beta$ be the image of $\alpha$ in $H^2(K,R)$. By assumption (i), one can then find some finite extensions $K_1,...,K_r$ of $K$ and an $r$-tuple $(x_1,...,x_r) \in \prod_{i=1}^r K_q^M(K_i)$ such that:
\[\begin{cases} 
\beta|_{K_i}=0\,\text{ for each }\,i\in \{1,...,r\}, \\
x = \prod_{i=1}^r N_{K_i/K}(x_i).
\end{cases}\]
One then observes that, for each $i$, one can lift the class $\alpha|_{K_i} \in H^2(K_i,Q)$ to a class $\gamma_i \in H^2(K_i,P)$. Since  $H^2(K_iK_{\infty},P) \rightarrow H^2(K_iK_{\infty},Q)$ is injective, $\gamma_i$ lies in fact in $\mathrm{Ker}\left( H^2(K_i,P) \rightarrow H^2(K_iK_{\infty},P)\right)$.  By assumption (ii), one can then find some finite extensions $K_{i,1},...,K_{i,r_i}$ of $K_i$ and an $r_i$-tuple $(x_{i,1},...,x_{i,r_i}) \in \prod_{j=1}^{r_i} K_q^M(K_{i,j})$ such that:
\[\begin{cases} 
\gamma_i|_{K_{i,j}}=0\,\text{ for each }\, j\in \{1,...,r_i\}, \\
x_i = \prod_{j=1}^{r_i} N_{K_{i,j}/K_i}(x_{i,j}).
\end{cases} \]
 Hence $x = \prod_{i=1}^{r} \prod_{j=1}^{r_i} N_{K_{i,j}/K}(x_{i,j})$, and $x \in N_q(\alpha/K)$.
\end{proof}

\subsubsection{Proof of Theorem \ref{galcoh}}\label{firststep3}

We are now ready to prove Theorem \ref{galcoh}.

\begin{proof}[Proof of Theorem \ref{galcoh}]
Consider a finite Galois module $P$ over $K$ which becomes diagonalizable over $K_{\infty}$ and a cohomology class $\alpha \in \mathrm{Ker}\left(H^2(K,P) \rightarrow H^2(K_{\infty},P) \right)$. We want to prove that $N_q(\alpha/K)=K_q^M(K)$.

Given a prime number $\ell$, one can write the exact sequence of Galois modules:
$$0 \rightarrow \ell P \rightarrow P \rightarrow P/\ell P \rightarrow 0.$$
Now note that, since $P$ is diagonalizable over $K_{\infty}$, one can write an isomorphism of Galois modules over $K_{\infty}$:
$$P|_{K_{\infty}} \cong P' \times \prod_{i=1}^r \mu_{\ell^{s_i}}$$
for some Galois module $P'$ of order prime to $\ell$ and for some positive integers $s_1,...,s_r$. For each finite extension $L$ of $K$, the morphism $H^2(LK_{\infty},\ell P) \rightarrow H^2(LK_{\infty},P)$ can therefore be identified with the morphism $H^2(LK_{\infty},P')\times \prod_{i=1}^r H^2(LK_{\infty},\mu_{\ell^{s_i-1}}) \rightarrow H^2(LK_{\infty},P')\times\prod_{i=1}^r H^2(LK_{\infty},\mu_{\ell^{s_i}})$, which is always injective. Hence, by carrying out an induction on the exponent of $P$ in which we repeatedly apply Lemma \ref{easy} and in which the field $K$ varies, we can assume that $P$ is $\ell$-torsion for some prime number $\ell$.

Now consider a finite Galois extension $L$ of $K$ such that $\alpha|_{L}=0$ and $\mathrm{Gal}(\overline{L}/L)$ acts trivially on $P$. Fix a prime number $p$ and let $L_p$ be a field fixed by a $p$-Sylow of $\mathrm{Gal}(L/K)$. Consider the two following cases:

\paragraph*{1st case: $p\neq\ell$.} Since $\alpha|_{L}=0$, we see by a restriction-corestriction argument that $\alpha|_{L_p}=0$. Hence $N_q(\alpha/K)$ contains $[L_p:K]K_q^M(K)$.
\paragraph*{2nd case: $p=\ell$.} Since $\mathrm{Gal}(L/L_p)$ is a $p$-group and since $P$ is an $\mathbb{F}_p$-vector space, we can find a basis of $P$ such that all the elements of $\mathrm{Gal}(L/L_p)$ act on $P$ via a unipotent upper-triangular matrix. This means that we have a d\'evissage of Galois modules over $L_p$:
\begin{gather*}
0 \rightarrow \mathbb{Z}/p\mathbb{Z} \rightarrow P \rightarrow P_1 \rightarrow 0\\
0 \rightarrow \mathbb{Z}/p\mathbb{Z} \rightarrow P_1 \rightarrow P_2 \rightarrow 0\\
...\\
0 \rightarrow \mathbb{Z}/p\mathbb{Z} \rightarrow P_{s-1} \rightarrow P_s \rightarrow 0,
\end{gather*}
with $P_s=\mathbb{Z}/p\mathbb{Z}$. Since $P$ becomes diagonalizable over $K_{\infty}$, all these exact sequences split over $L_pK_{\infty}$. Hence, by applying Corollary \ref{sevbr} and Lemma \ref{easy}, it follows that $N_q(\alpha|_{L_p}/L_p)=K_q^M(L_p)$. Hence $N_q(\alpha/L)$ contains $[L_p:L]K_q^M(L)$.\\

Thus $N_q(\alpha/L)$ contains $[L_p:L]K_q^M(L)$ for every prime $p$, which implies that $N_q(\alpha/L)=K_q^M(L)$.
\end{proof}

\subsection{Second step: Principal homogeneous spaces}

In this second step, we prove theorem \ref{reduction}(i) for principal homogeneous spaces under smooth connected linear groups as well as theorem \ref{reduction}(ii). We start with the characteristic 0 case.

\begin{proposition}\label{split}
Let $q$ be a non-negative integer. Let $K$ be a field of characteristic 0 and with cohomological dimension at most $q+1$. Then for any principal homogeneous space $Z$ under a quasi-split reductive $K$-group $G$, we have $N_{q}(Z/K)=K_q^M(K)$.
\end{proposition}

\begin{proof}[Proof]
Let $Z$ be a principal homogeneous space under a quasi-split reductive $K$-group $G$. Let $z$ be the class of $Z$ in $H^1(K,G)$. By Theorem 11.1 of \cite{Ste}, we have:
$$H^1(K,G) = \bigcup_{T\subseteq G} \mathrm{Im}\left( H^1(K,T) \rightarrow H^1(K,G) \right) $$
where $T$ runs over the maximal $K$-tori of $G$. Let then $T_z$ be a maximal $K$-torus of $G$ such that $z$ can be lifted to a class $\tilde{z} \in H^1(K,T_z)$. By Ono's lemma (Theorem 1.5.1 in \cite{Ono}), we can find an exact sequence:
$$0\rightarrow F \rightarrow R_0 \rightarrow T_z^n \times R_1 \rightarrow 0$$
for some positive integer $n$, some quasi-trivial tori $R_0$ and $R_1$ and some finite commutative algebraic group $F$. By considering the associated cohomology exact sequence, we get an injection:
$$H^1(K,T_z)^n \hookrightarrow H^2(K,F).$$
If $\alpha\in H^2(K,F)$ stands for the image of the $n$-tuple $(\tilde{z},...,\tilde{z})\in H^1(K,T_z)^n$, then $N_q(\alpha/K)$ is contained in $N_{q}(Z/K)$. But by corollary \ref{galcoh2}, we have $N_q(\alpha/K)=K_q^M(K)$. Hence $N_{q}(Z/K)=K_q^M(K)$.
\end{proof}

\begin{theorem}\label{phs}
Let $q$ be a non-negative integer. Let $K$ be a field of characteristic 0 and with cohomological dimension at most $q+1$. Then for any principal homogeneous space $Z$ under a connected linear $K$-group $G$, we have $N_{q}(Z/K)=K_q^M(K)$.
\end{theorem}

\begin{proof}
Let $G$ a connected linear $K$-group, $U$ its unipotent radical and $Z$ a principal homogeneous space under $G$. Then $Z'=Z/U$ is naturally a $G/U$-torsor and, since $H^1$ is trivial for unipotent groups (cf.~Proposition 6 of \cite[III.2.1]{GC}), we know that every fiber of $Z\to Z'$ (which is a $U$-torsor) has rational points. Then it is easy to see that the result holds for $Z$ as soon as it holds for $Z'$.

Assume then that $G$ is reductive and let $Z$ be as above. Then there exists a quasi-split reductive $K$-group $H$ and a class $[a]\in H^1(K,H)$ such that $G={_a}H$ (Proposition 16.4.9 of \cite{springer}). Fix an element $x$ in $K_q^M(K)$. By Proposition \ref{split}, we can find some finite extensions $K_1,...,K_r$ of $K$ and an $r$-tuple $(x_1,...,x_r) \in \prod_{i=1}^r K_q^M(K_i)$ such that:
\[\begin{cases} 
[a]|_{K_i}=0\text{ for each }i\in \{1,...,r\},\\
x = \prod_{i=1}^r N_{K_i/K}(x_i).
\end{cases}\]
Hence for each $i$, the reductive group $G_{K_i}$ is quasi-split. By applying Proposition \ref{split} once again, we can find some finite extensions $K_{i,1},...,K_{i,r_i}$ of $K_i$ and an $r_i$-tuple $(x_{i,1},...,x_{i,r_i}) \in \prod_{j=1}^{r_i} K_q^M(K_{i,j})$ such that:
\[\begin{cases} 
Z(K_{i,j})\neq\emptyset\text{ for each }j\in \{1,...,r_i\},\\
x_i = \prod_{j=1}^{r_i} N_{K_{i,j}/K_i}(x_{i,j}).
\end{cases}\]
Hence $x = \prod_{i=1}^{r} \prod_{j=1}^{r_i} N_{K_{i,j}/K}(x_{i,j})$, and $x \in N_q(Z/K)$.
\end{proof}

We now deal with fields of positive characteristic.

\begin{theorem} \label{poscar}
Let $q$ be a non-negative integer and let $K$ be any field of characteristic $p>0$. 
\begin{itemize}
\item[(i)] If $K$ is perfect and has cohomological dimension at most $q+1$, then for any principal homogeneous space $Z$ under a smooth linear connected group over $K$, we have $N_q(Z/K)=K_q^M(K)$.
\item[(ii)] If $K$ is imperfect and all its finite extensions have separable cohomological dimension at most $q+1$, we have $N_q(Z/K)=K_q^M(K)$ for any principal homogeneous space $Z$ under a reductive group over $K$.
\end{itemize}

\end{theorem}

\begin{proof}
We start by proving (i). Let $Z$ be a principal homogeneous space under a smooth linear connected $K$-group $G$. Let also $\{u_1,...,u_q\}$ be a symbol in $K_q^M(K)$. We want to prove that $\{u_1,...,u_q\}\in N_q(Z/K)$.

Since $K$ is a perfect field, the triviality of $H^1$ for unipotent groups still holds in this case. Then, by proceeding as in the proofs of Proposition \ref{split} and Theorem \ref{phs}, we may and do assume that $G$ is a torus $T$. We can then find an isotrivial torus $\tilde{\mathcal{T}}$ on the ring of Witt vectors $W(K)$ which lifts $T$. Moreover, since the map $H^1(W(K),\tilde{\mathcal{T}}) \rightarrow H^1(K,T)$ is an isomorphism (Theorem 11.7 of \cite{gro}), we can find a lifting $\tilde{\mathcal{Z}}$ of $Z$ to $W(K)$. Let $\tilde{T}$ and $\tilde{Z}$ be the generic fibers of $\tilde{\mathcal{T}}$ and $\tilde{\mathcal{Z}}$ respectively: they are defined on the fraction field $\tilde{K}$ of $W(K)$.

By Theorem 3 of \cite{Kato}, the cohomological dimension of $\tilde{K}$ is at most $q+2$. Hence, by Theorem \ref{phs}, we know that $N_{q+1}(\tilde{Z}/\tilde{K})=K_{q+1}^M(\tilde{K})$. We can therefore find $r$ finite extensions $\tilde{K}_1,...,\tilde{K}_r$ of $\tilde{K}$ and an $r$-tuple $(x_1,...,x_r) \in \prod_{i=1}^r K_q^M(\tilde{K}_i)$ such that:
$$\begin{cases} 
\tilde{Z}(\tilde{K}_i)\neq \emptyset\text{ for each }i\in \{1,...,r\}, \\
\{p,\tilde{u}_1,...,\tilde{u}_q\} = \prod_{i=1}^r N_{\tilde{K}_i/\tilde{K}}(x_i)
 \end{cases} $$  
 for some liftings $\tilde{u}_1,...,\tilde{u}_q\in W(K)^{\times}$ of $u_1,...,u_q$. Denote by $k_1,...,k_r$ the residue fields of $\tilde{K}_1,...,\tilde{K}_r$. By using the compatibility of the norm morphism in Milnor $K$-theory with the residue maps (Proposition 7.4.1 of \cite{GS}), we deduce that $\{u_1,...,u_q\}$ is a product of norms coming from the $k_i$'s. Moreover, for each $i\in \{1,...,n\}$, the restriction morphism $H^1(\mathcal{O}_{\tilde{K}_i},\tilde{\mathcal{T}}) \rightarrow H^1(\tilde{K}_i,\tilde{T})$ can be identified with the inflation morphism $H^1(\tilde{K}_i^{nr}/\tilde{K}_i,H^0(\tilde{K}_i^{nr},\tilde{T})) \rightarrow H^1(\tilde{K}_i,\tilde{T})$ and is therefore injective. This implies that $Z(k_i)\neq \emptyset$, so that $\{u_1,...,u_q\}\in N_q(Z/K)$ and (i) is proved.
 
 We now prove (ii). Let $Z$ be a principal homogeneous space under a reductive  $K$-group $G$. There is an exact sequence:
 $$1 \rightarrow H \rightarrow G \rightarrow S \rightarrow 1$$
 in which $H$ is semisimple and $S$ is an isotrivial torus. Such a sequence induces a cohomology exact sequence:
 $$H^1(K,H) \rightarrow H^1(K,G) \rightarrow H^1(K,S).$$
 Hence, by using a variant of Lemma \ref{easy} for non-abelian cohomology, we can reduce to the case where $G$ is a semisimple group or an isotrivial torus. But the case where $G$ is semisimple can itself be reduced to the case where $G$ is an isotrivial torus by proceeding as in the proofs of Proposition \ref{split} and Theorem \ref{phs}, by observing that Theorem 11.1 of \cite{Ste} still holds for semisimple groups over imperfect fields (see Theorem 2 of \cite{Serre}) and by using the isotriviality of maximal tori of semisimple groups. We henceforth assume that $G$ is an isotrivial torus $T$. By adopting the same notations as in part (i) and by proceeding exactly in the same way, we only need to prove that $N_{q+1}(\tilde{Z}/\tilde{K}) =K_{q+1}^M(\tilde{K})$. 
 
 Here, the field $\tilde{K}$ need not have cohomological dimension at most $q+2$, and hence we cannot directly use Theorem \ref{phs}.
 But we know that the torus $\tilde{T}$ is unramified. Hence, by Hilbert's Theorem 90, we know that $[Z] \in \mathrm{Ker}\left( H^1(\tilde K,\tilde{T}) \rightarrow H^1(\tilde K^{nr},\tilde{T})\right)$. By using Ono's lemma just as in Proposition \ref{split}, we see that we only need to prove that, for any finite Galois module $M$ over $K$ which becomes diagonalisable over $K^{nr}$ and any cohomology class $\alpha \in \mathrm{Ker}\left(  H^2(\tilde{K},M)\rightarrow H^2(\tilde{K}^{nr},M)\right)$, we have $N_{q+1}(\alpha/\tilde{K})=K_{q+1}^M(\tilde{K})$. But since all finite extensions of $K$ have separable cohomological dimension at most $q+1$, Lemma 1 of \cite{gille} and Theorem 3(1) of \cite{Kato} imply that the restriction map $H^{q+3}(L,\mu_p^{\otimes (q+2)}) \rightarrow H^{q+3}(L^{nr},\mu_p^{\otimes (q+2)})$ is injective for every finite extension $L$ of $\tilde{K}$. One can therefore apply Theorem \ref{galcoh} with $\tilde{K}_{\infty}=\tilde{K}^{nr}$, provided that one checks the following lemma.
\end{proof}

\begin{lemma}
Let $K$ be a complete discrete valuation field of mixed characteristic $(0,p)$. Let $k$ be the residue field of $K$ and assume that $K$ contains a primitive $p$-th root of unity. Then the group ${_p}\br(K^{nr}/K)$ is spanned by cyclic central simple algebras.
\end{lemma}

\begin{proof}
By exercise 3 of section XIII.3 of \cite{corpsloc}, we have the exact sequence:
$$0 \rightarrow \br\, k \rightarrow \br(K^{nr}/K) \rightarrow H^1(k,\mathbb{Q}/\mathbb{Z}) \rightarrow 0.$$
By Albert's theorem (Theorem 9.1.8 of \cite{GS}), every element in ${_p}\br\, k$ is represented by a cyclic algebra. Moreover, if $\chi$ is an element in $H^1(k,\mathbb{Q}/\mathbb{Z})$, then a lifting of $\chi$ in $\br(K^{nr}/K)$ is given by the cyclic algebra $(\tilde{\chi},\pi)$ where $\tilde{\chi}\in H^1(K,\mathbb{Q}/\mathbb{Z})$ is the unramified lifting of $\chi$ and $\pi$ is a uniformizer in $K$. Hence every element in ${_p}\br(K^{nr}/K)$ is a sum of at most two cyclic algebras.
\end{proof}

\begin{remarque}
With these two first steps, we have proved that a perfect field $K$ of cohomological dimension at most $q+1$ is $C_{\PHS}^q$. We have also proved part (ii) of Theorem \ref{reduction}.
\end{remarque}

\subsection{Third step: Homogeneous spaces}\label{last}

We finally prove part (i) of Theorem \ref{reduction}. For this purpose, recall the following theorem of Springer (Theorem 3.4 of \cite{SpringerH2}):

\begin{theorem}
Let $K$ be a perfect field with algebraic closure $\bar K$, let $L$ be a $K$-kernel with underlying (smooth) $\bar K$-algebraic group $\bar G$. Then for every $\eta\in H^2(K,L)$, there exists a (smooth) finite nilpotent $\bar K$-subgroup $\bar F\subset\bar G$, a $K$-kernel $F$ with underlying $\bar K$-group $\bar F$ and an injective morphism of $K$-kernels $F\to L$ compatible with the inclusion $\bar F\subset\bar G$ such that the induced relation $H^2(K,F)\multimap H^2(K,L)$ has $\eta$ in its image.
\end{theorem}

\begin{proof}[Proof of Theorem \ref{reduction}]
Let $K$ be a perfect field of cohomological dimension at most $q+1$ and $Z$ a homogeneous space under a smooth linear connected $K$-group $G$. We claim that we may assume that $Z$ has smooth stabilizers (which of course is automatic in characteristic 0). Indeed, consider the Frobenius twist $Z^{(p)}$ and the Frobenius morphism $Z\to Z^{(p)}$, which is surjective since $Z$ is smooth. Since the Frobenius twist is functorial and it \emph{does not} modify rational points (cf.~Expos\'e VII\textsubscript{A}, \S4.1 of \cite{SGA3}), we may replace $G$ by $G/_\text{Fr}G$, where $_\text{Fr}G$ denotes the kernel of $G\to G^{(p)}$. The stabilizers will be thus replaced by the corresponding quotient by the Frobenius kernel. By Proposition 8.3 in \textit{loc.~cit.} we know that some power of this construction will give smooth stabilizers.

Now, following Section 2.3 in \cite{DLA}, we associate to $Z$ a gerb $\M$ and an injective morphism of gerbs $\M\to\tors(G)$, where $\tors(G)$ is the trivial gerb of torsors under $G$. The gerb $\M$ represents a class $\eta\in H^2(K,L)$ for some $K$-kernel $L$ with smooth underlying $\bar K$-group. Springer's Theorem tells us then that there exists an injective morphism of gerbs $\M_F\to\M$ with $\M_F$ a finite gerb whose underlying group $\bar F$ is smooth and nilpotent. By Proposition 3.2 from \cite{DLA}, this implies the existence of a homogeneous space $Z_F$ under $G$ with finite nilpotent stabilizers and a (a fortiori surjective) $G$-morphism $Z_F\to Z$. Since clearly $N_{K'/K}(Z_F/K)\subset N_{K'/K}(Z/K)$, we may and do assume that the stabilizers of $Z$ are finite and solvable.\\

Consider then the gerb $\M$ with underlying (smooth, finite, solvable) $\bar K$-group $\bar F$. It represents a class $\eta\in H^2(K,F)$ for some finite $K$-kernel $F$. We proceed by induction on the order of $\bar F$, noting that the case of order 1 corresponds to $G$-torsors, which are already dealt with by Theorem \ref{phs}.

Denote by $F'=[F,F]$ the derived subkernel of $F$. It is well-defined since the underlying $\bar K$-subgroup $\bar F'$ is characteristic in $\bar F$. By the ``functoriality'' of non-abelian $H^2$ for surjective morphisms (cf.~\S2.2.3 in \cite{DLA}), we obtain from $\eta\in H^2(K,F)$ a class $\eta^\ab\in H^2(K,F^\ab)$, where $F^\ab$ is the finite abelian $K$-group ($=K$-kernel) naturally obtained from $F/F'$ (cf.~for instance \S1.15 of \cite{FSS}).

Choose an element $x\in K_q^M(K)$. By Corollary \ref{galcoh2}, we can find some finite extensions $K_1,...,K_r$ of $K$ and an $r$-tuple $(x_1,...,x_r) \in \prod_{i=1}^r K_q^M(K_i)$ such that:
\[\begin{cases} 
\eta^\ab|_{K_i}=0\text{ for each }i\in \{1,...,r\}, \\
x = \prod_{i=1}^r N_{K_i/K}(x_i).
\end{cases}\]
Using for instance the cocyclic approach to the classes in $H^2(K,F)$, we immediately see that the triviality of $\eta^\ab|_{K_i}$ implies that $\eta|_{K_i}$ comes from $H^2(K_i,F')$. This implies in turn that there exist $K_i$-gerbs $\M_i$ with underlying kernel $F'$ and injective morphisms of gerbs $\M_i\to\M$ which, composed with $\M\to\tors(G)$, prove by Proposition 3.2 in \cite{DLA} the existence of $G_{K_i}$-homogeneous spaces $Z_i$ with geometric stabilizers isomorphic to $\bar F'$ and $G_{K_i}$-equivariant morphisms $Z_i\to Z_{K_i}$. By the inductive assumption, we deduce that $N_q(Z_i/K_i)=K_q^M(K_i)$. We can therefore find some finite extensions $K_{i,1},...,K_{i,r_i}$ of $K_i$ and an $r_i$-tuple $(x_{i,1},...,x_{i,r_i}) \in \prod_{j=1}^{r_i} K_q^M(K_{i,j})$ such that:
\[\begin{cases} 
Z_i(K_{i,j})\neq \emptyset\text{ for each }j\in \{1,...,r_i\}, \\
x_i = \prod_{j=1}^{r_i} N_{K_{i,j}/K_i}(x_{i,j}).
\end{cases}\]
Hence $x = \prod_{i=1}^{r} \prod_{j=1}^{r_i} N_{K_{i,j}/K}(x_{i,j})$, and $x \in N_q(Z/K)$.
\end{proof}

\begin{remarque}
In the case of fields of characteristic 0, there is a second proof of our Main Theorem that completely avoids the use of gerbs by using the main result in \cite{DLA}. This result allows to reduce the proof to the cases of homogeneous spaces of $\mathrm{SL}_n$ with finite stabilizer and homogeneous spaces of semisimple simply connected groups with stabilizers of ``ssu-type''. There is still some basic non-abelian cohomology to be dealt with in the latter case, but it is much less technical (the point being that all the gerbs are hidden in the results by Demarche and the second author). The case of $\mathrm{SL}_n$ and finite stabilizers is easily reduced to the case of solvable stabilizers by a restriction-corestriction argument using the Sylow subgroups of the stabilizers. Then one can follow the techniques on universal torsors used by Harpaz and Wittenberg in \cite{hw} in order to get a proof by induction that completely avoids non-abelian cohomology.
\end{remarque}

\textit{Acknowledgements.} We thank Olivier Wittenberg, Tam\'as Szamuely and the anonymous referees for very useful comments. The second author's research was partially supported by Conicyt via Fondecyt Grant 11170016 and PAI Grant 79170034.

\end{document}